\documentclass[10pt,final]{article}

\usepackage{amssymb,amsmath,amsthm,authblk,graphicx,subfig,url,mathptmx}

\newtheoremstyle{theorem}
  {10pt}		  
  {10pt}  
  {\sl}  
  {\parindent}     
  {\bf}  
  {. }    
  { }    
  {}     
\theoremstyle{theorem}
\newtheorem{theorem}{Theorem}

\newtheoremstyle{defi}
  {10pt}		  
  {10pt}  
  {\rm}  
  {\parindent}     
  {\bf}  
  {. }    
  { }    
  {}     
\theoremstyle{defi}

\newtheorem{example}{Example}

\begin{document}

\author[1]{B. Altunkaya\thanks{bulent.altunkaya@ahievran.edu.tr}}
\author[2]{L. Kula\thanks{lkula@ahievran.edu.tr}}
\affil[1]{Department of Mathematics, Faculty of  Education, University of Ahi Evran, K{\i}r\c{s}ehir, Turkey}
\affil[2]{Department of Mathematics, Faculty of  Science, University of Ahi Evran, K{\i}r\c{s}ehir, Turkey}

\title{Some characterizations of slant and spherical helices due to sabban frame}
\date{}
\maketitle

\begin{abstract}

In this paper, we are investigating that under which conditions of the geodesic curvature of unit speed curve $\gamma$ that lies on the unit sphere, the curve $c$ which is obtained by using $\gamma$,  is a spherical helix or slant helix. 

\end{abstract}


\section{Introduction}
\label{sec:intro}

Izumiya and Takeuchi [4],[2] have defined helices, spherical helices, slant helices and conical geodesic curve and given a classification of special developable surfaces under the condition of the existence of such  a special curve as a geodesic [1].

Encheva and Georgiev [3] have used a similar method and determined a $Frenet$ $curve$ up to a direct similarity of $R^3$.

In this paper we have used the the method in [2], [3], and [4] to construct spherical helices and slant helices.

In section 2, we recall some basic concepts of differential geometry of space curves that we will use later. In the next section we will provide some new theorems and proofs about the construction of  spherical helices and slant helices.

\section{Basic Concepts}
\label{sec:meth} 

We now recall some basic concepts on classical differential geometry of curves in the Euclidean space  $ E^3$. For a regular curve $ c :I\subset R\rightarrow { R }^{ 3 }$  with curvature and torsion, $ \kappa$  and $ \tau $, the following Frenet-Serret formulae are given in [5] written in the matrix form for  $ \kappa>0$
\begin{equation*}
\left[ \begin{matrix} { T }^{ ' } \\ { N }^{ ' } \\ { B }^{ ' } \end{matrix} \right] =\left[ \begin{matrix} 0 & \kappa \nu  & 0 \\ -\kappa \nu  & 0 & \tau \nu  \\ 0 & -\tau \nu  & 0 \end{matrix} \right] \left[ \begin{matrix} T \\ N \\ B \end{matrix} \right] 
\end{equation*}
where 
\begin{equation}
\nu =\nu (t)=\left\| { c }^{ ' }(t) \right\|,
\kappa =\kappa (t)=\frac { \left\| { c }^{ ' }(t)\times { c }^{ '' }(t) \right\|  }{ { \left\| { c }^{ ' }(t) \right\|  }^{ 3 } }, 
\tau =\tau (t)=\frac { \left< { c }^{ ' }(t)\times { c }^{ '' }(t),{ c }^{ ''' }(t) \right>  }{ { \left\| { c }^{ ' }(t)\times { c }^{ '' }(t) \right\|  }^{ 2 } } 
\end{equation}
and 

\begin{equation}
T=T(t)=\frac { { c }^{ ' }(t) }{ \left\| { c }^{ ' }(t) \right\|  } ,B=B(t)=\frac { { c }^{ ' }(t)\times { c }^{ '' }(t) }{ \left\| { c }^{ ' }(t)\times { c }^{ '' }(t) \right\|  } ,N=N(t)=B(t)\times T(t)
\end{equation}
In the formulae above, we denote unit tangent vector with $T$, binormal unit vector with $B$,  unit principal normal vector with $N$, cross product with $\times$, and inner product with $<,>$.

A regular curve $c$ with $ \kappa>0$ is a cylindrical helix if and only if the ratio
\begin{equation*}
 \frac { \tau  }{ \kappa  } 
\end{equation*}
 is constant [5]. 

A regular curve $c$ with $ \kappa>0$ is a slant helix if and only if the geodesic curvature of the spherical image of the principal normal indicatrix of  $c$ 
\begin{equation}
\sigma(t)=\left( \frac { { \kappa  }^{ 2 } }{ { { \nu \left( { \kappa  }^{ 2 }+{ \tau  }^{ 2 } \right)  } }^{ { 3 }/{ 2 } } } { \left( \frac { \tau  }{ \kappa  }  \right)  }^{ ' } \right) \left( t \right) 
\end{equation}
is constant [4].

A regular curve $c$ with $ \kappa>0$ lies on the surface  of a sphere which has a radius  $r$ if and only if 
\begin{equation}
 { r }^{ 2 }=\left( \frac { 1 }{ { \kappa  }^{ 2 } } +{ \left( \frac { 1 }{ \nu \tau  } { \left( \frac { 1 }{ \kappa  }  \right)  }^{ ' } \right)  }^{ 2 } \right) \left( t \right) 
\end{equation}
satisfies [6]. We can easily simplify this equation as follows
\begin{equation}
\left( \frac { 1 }{ \nu  } \left[ \frac { 1 }{ \nu \tau  } { \left( \frac { 1 }{ \kappa  }  \right)  }^{ ' } \right] ^{ ' }+\frac { \tau  }{ \kappa  }  \right) \left( t \right) =0.
\end{equation}

Let $\gamma :I\rightarrow { S }^{ 2 }$ be a unit speed spherical curve with an arc length parameter $s$ and denote ${ \gamma  }^{ ' }\left( s \right) =t\left( s \right)   $ where ${ \gamma  }^{ ' }\left( s \right)=\frac { d\gamma  }{ ds }  $. If we set a vector $p\left( s \right) =\gamma \left( s \right) \times t\left( s \right)$, by definition we have an orthonormal frame $\left\{ \gamma \left( s \right), t\left( s \right) , p\left( s \right)  \right\} $ along $\gamma$. This frame is called the $Sabban$  $frame$ of $\gamma$. Then we have the following $Frenet-Serret$ formulae of  $\gamma$ 
\begin{equation}
\begin{matrix} { \gamma  }^{ ' }\left( s \right) =t\left( s \right)  \\ { t }^{ ' }\left( s \right) =-\gamma \left( s \right) +{ k }_{ g }\left( s \right) p\left( s \right)  \\ { p }^{ ' }\left( s \right) =-{ k }_{ g }\left( s \right) t\left( s \right)  \end{matrix}
\end{equation}
where ${ k }_{ g }\left( s \right)$ is the $geodesic$  $curvature$ of the curve  $\gamma$ on $S^2$  which is ${ k }_{ g }\left( s \right) =det\left( \gamma \left( s \right) ,t\left( s \right) , { t }^{ ' } \left( s \right)  \right) $ [2].

In [2] Izuyama and Takeuchi showed a way to construct all $Bertrand$ $curves$ by the following formula
\begin{equation}
c\left( s \right) =b\int _{ { s }_{ 0 } }^{ s }{ \gamma \left( \varphi  \right) d } \varphi +b\cot { \theta  } \int _{ { s }_{ 0 } }^{ s }{ p\left( \varphi  \right) d } \varphi +a
\end{equation}
where $b,\theta$ are constant numbers, a is a constant vector, and $\gamma$ is a unit speed curve on $S^2$ with the $Sabban$ $frame$ above. Also they showed that the spherical curve $\gamma$ is a circle if and only if the corresponding $Bertrand$ $curves$ are  circular helices.

In [3] Encheva and  Georgiev showed a way to construct all $Frenet$ $curves\left( \kappa >0 \right) $ by the following formula
\begin{equation}
c\left( s \right) =b\int { { e }^{ \int { k\left( s \right) ds }  }\gamma \left( s \right)  } ds+a
\end{equation}
where $b$ is a constant number, a is a constant vector, $\gamma$ is a unit speed curve on $S^2$ with the $Sabban$ $frame$ above, and $k:I\rightarrow R\quad$is a function of class $C^1$. Also they showed that the spherical curve $\gamma$ is a circle if and only if the corresponding $Frenet$  $curves$ are cylindrical helices.

If we use the equations in (1), (2), (6), (7), and (8) we can easily see the $Frenet$ $frame$ $\left\{ T, N, B \right\}$ of the curve $c$ and the $Sabban$ $frame$ $\left\{ \gamma, t, p \right\}$ of the curve $\gamma$  coincides. Therefore we can say the $tangent$  $indicatrix$ of the curve $c$ is $\gamma$.

\section{Special Results}

Now, we can deduce some results from the equations above. First, we want to show, under which circumstances the equation (8) is a spherical helix. As we know before, if $\gamma$ is a circle, the geodesic curvature of it is constant. Therefore we can write the theorem below.

\begin{theorem}
If the curve $\gamma$ is a circle, the curve c defined by (8) is a spherical helix if and only if  the function $k\left( s \right) =-{ k }_{ g }tan\left[ \left( { k }_{ g } \right) \left( s-{ b}_{ 1 } \right)  \right]$ where $b_{ 1 }\in R$ .
\end{theorem}

\begin{proof}
For the curve
\begin{equation*}
c\left( s \right) =b\int { { e }^{ \int { k\left( s \right) ds }  } \gamma \left( s \right)  } ds+a
\end{equation*}
If we calculate $\kappa$, $\tau$, and $\nu$ of the curve $c$ by using the equations at (1), we will find 
\begin{equation}
\begin{matrix} \kappa \left( s \right) =\frac { 1 }{ b{ e }^{ \int { k\left( s \right) ds }  } }  \\ \tau \left( s \right) =\frac { { k }_{ g }\left( s \right)  }{ b{ e }^{ \int { k\left( s \right) ds }  } }  \\  \nu \left( s \right) =b{ e }^{ \int { k\left( s \right) ds }  } \end{matrix}.
\end{equation}
Now, by putting these equations in (5), we have
\begin{equation*}
\left( \frac { 1 }{ \nu  } \left[ \frac { 1 }{ \nu \tau  } { \left( \frac { 1 }{ \kappa  }  \right)  }^{ ' } \right] ^{ ' }+\frac { \tau  }{ \kappa  }  \right) \left( s \right) =0
\end{equation*}
\begin{equation*}
\left( \frac { 1 }{ b{ e }^{ \int { kds }  } } \left[ \frac { 1 }{ b{ e }^{ \int { kds }  }\frac { { k }_{ g } }{ b{ e }^{ \int { kds }  } }  } { \left( \frac { 1 }{ \frac { 1 }{ b{ e }^{ \int { kds }  } }  }  \right)  }^{ ' } \right] ^{ ' }+\frac { \frac { { k }_{ g } }{ b{ e }^{ \int { kds }  } }  }{ \frac { 1 }{ b{ e }^{ \int { kds }  } }  }  \right) \left( s \right) =0
\end{equation*}
\begin{equation*}
\left( \frac { 1 }{ b{ e }^{ \int { kds }  } } \left[ \frac { 1 }{ { k }_{ g } } { \left( b{ e }^{ \int { kds }  } \right)  }^{ ' } \right] ^{ ' }+{ k }_{ g } \right) \left( s \right) =0
\end{equation*}
\begin{equation*}
\left( \frac { 1 }{ { k }_{ g }{ e }^{ \int { kds }  } } \left[ { k }^{ ' }{ e }^{ \int { kds }  }+{ k }^{ 2 }{ e }^{ \int { kds }  } \right] +{ k }_{ g } \right) \left( s \right) =0
\end{equation*}
\begin{equation*}
{ k }^{ ' }\left( s \right) +{ k }^{ 2 }\left( s \right) =-{ { k }_{ g } }^{ 2 }.
\end{equation*}
If we solve this differential equation, we will have
\begin{equation*}
k\left( s \right) =-{ k }_{ g }tan\left[ \left( { k }_{ g } \right) \left( s-{ b }_{ 1 } \right)  \right] 
\end{equation*}
Conversely,  if we take $k\left( s \right) =-{ k }_{ g }tan\left[ \left( { k }_{ g } \right) \left( s-{ b }_{ 1 } \right)  \right]$ in (8) then
\begin{equation*}
\int { k\left( s \right) ds=\int { -{ k }_{ g }tan\left[ \left( { k }_{ g } \right) \left( s-{ b }_{ 1 } \right)  \right] ds }  }.
\end{equation*}
 Let $u={ k }_{ g }\left( s-{ b }_{ 1 } \right) ={ k }_{ g }s-{ k }_{ g }{ b }_{ 1 }$ then ${ k }_{ g }ds=du$, by using these equations
\begin{eqnarray*}
&& {\int k\left( s \right) ds } =\int { -\tan { u } du } \\
&& \quad\quad \quad\quad   =\ln { \cos { u }  } +\ln { { b }_{ 2 } } \\
&&  \quad\quad \quad\quad =\ln { \left[ { b }_{ 2 }\cos { \left\{ { k }_{ g }\left( s-{ b }_{ 1 } \right)  \right\}  }  \right]  }
\end{eqnarray*}
we have
\begin{eqnarray*}
&& c\left( s \right) =b\int { { e }^{ \int { k\left( s \right) ds }  } \gamma \left( s \right)  } ds+a\\
&&\quad\quad  =b\int { { e }^{ \int {  -{ k }_{ g }tan\left[ \left( { k }_{ g } \right) \left( s-{ b }_{ 1 } \right)  \right] ds }  } \gamma \left( s \right)  } ds+a\\
&&\quad\quad=b\int { { e }^{ \ln { \left[ { b }_{ 2 }\cos { \left\{ { k }_{ g }\left( s-{ b }_{ 1 } \right)  \right\}  }  \right]  } } \gamma \left( s \right)  } ds+a\\
&&\quad\quad=b\int { { b }_{ 2 }\cos { \left\{ { k }_{ g }\left( s-{ b }_{ 1 } \right)  \right\}  }  \gamma \left( s \right)  } ds+a.
\end{eqnarray*}
where $b_1,b_2\in R$.

Now, we must show that curve $c$ is spherical. If we use (4) to do it, we will have
\begin{eqnarray*}
&&{ r }^{ 2 }=\left( \left( \frac { 1 }{ { \kappa  }^{ 2 } } +{ \left( \frac { 1 }{ \nu \tau  } { \left( \frac { 1 }{ \kappa  }  \right)  }^{ ' } \right)  } \right) ^{ 2 } \right) \left( s \right) \\ 
&&=\left( { b }^{ 2 }{ e }^{ 2\int { kds }  }+\left( \frac { 1 }{ b{ e }^{ \int { kds }  }\frac { { k }_{ g } }{ b{ e }^{ \int { kds }  } }  } { \left( \frac { 1 }{ \frac { 1 }{ b{ e }^{ \int { kds }  } }  }  \right)  }^{ ' } \right) ^{ 2 } \right) \left( s \right) \\ 
&&=\left( { b }^{ 2 }{ e }^{ 2\int { kds }  }+\left( \frac { 1 }{ { k }_{ g } } { \left( b{ e }^{ \int { kds }  } \right)  }^{ ' } \right) ^{ 2 } \right) \left( s \right) \\ 
&&=\left( { b }^{ 2 }{ e }^{ 2\int { kds }  }+\frac { b^{ 2 }k^{ 2 } }{ { k }_{ g }^{ 2 } } { e }^{ 2\int { kds }  } \right) \left( s \right) \\ 
&&=\left( { b }^{ 2 }{ e }^{ 2\int { kds }  }\left( 1+\frac { k^{ 2 } }{ { k }_{ g }^{ 2 } }  \right)  \right) \left( s \right) \\ 
&&={ b }^{ 2 }{ b_{ 2 } }^{ 2 }\cos ^{ 2 }{ \left\{ { k }_{ g }\left( s-b_{ 1 } \right)  \right\}  } \left( 1+\frac { \left( -{ k }_{ g }tan\left[ \left( { k }_{ g } \right) \left( s-{ b }_{ 1 } \right)  \right]  \right) ^{ 2 } }{ { k }_{ g }^{ 2 } }  \right) \\ 
&&={ b }^{ 2 }{ b_{ 2 } }^{ 2 }\cos ^{ 2 }{ \left\{ { k }_{ g }\left( s-b_{ 1 } \right)  \right\}  } \left( \frac { 1 }{ \cos ^{ 2 }{ \left\{ { k }_{ g }\left( s-b_{ 1 } \right)  \right\}  }  }  \right) \\
&& =b^{ 2 }{ b_{ 2 } }^{ 2 }.
\end{eqnarray*}
Therefore, we can say curve $c$ lies on a sphere which has a radius $\left| bb_2 \right| $
\end{proof}

\begin{example}
Let's take $\gamma \left( s \right) =\left\{ \frac { 1 }{ \sqrt { 3 }  } cos\left( \frac { s }{ \left( { 1 }/{ \sqrt { 3 }  } \right)  }  \right) ,\frac { 1 }{ \sqrt { 3 }  } sin\left( \frac { s }{ \left( { 1 }/{ \sqrt { 3 }  } \right)  }  \right) ,\sqrt { \frac { 2 }{ 3 }  }  \right\}$, we know that $\gamma$ is a spacelike curve on $S^2$ with the geodesic curvature $\sqrt { 2 }$. Then due to Theorem 1,  

\begin{equation*}
k\left( s \right) ={ k }_{ g }tan\left[ \left( { k }_{ g } \right) \left( s-{ b}_{ 1 } \right)  \right] 
\end{equation*}
and
\begin{equation*}
\alpha\left( s \right) =b\int { { b }_{ 2 }\cos { \left\{ { k }_{ g }\left( s-{ b }_{ 1 } \right)  \right\}  }  \gamma \left( s \right)  } ds+a
\end{equation*}
where $b,b_{ 1 },b_{ 2 }\in R$. If we take $b=2, b_{ 1 }=0 ,b_{ 2 }=1$ then we have
\begin{equation*}
\begin{matrix} 
{ \alpha  }_{ 1 }\left( s \right) =-2\sqrt { \frac { 2 }{ 3 }  } \cos { \left( \sqrt { 3 } s \right)  } \sin { \left( \sqrt { 2 } s \right)  } +2\cos { \left( \sqrt { 2 } s \right)  } \sin { \left( \sqrt { 3 } s \right)  }  \\ { \alpha  }_{ 2 }\left( s \right) =-\frac { 2 }{ 3 } \left( 3\cos { \left( \sqrt { 2 } s \right)  } \cos { \left( \sqrt { 3 } s \right)  } +\sqrt { 6 } \sin { \left( \sqrt { 2 } s \right)  } \sin { \left( \sqrt { 3 } s \right)  }  \right) \\ { \alpha  }_{ 3 }\left( s \right) =\frac { 2\sin { \left( \sqrt { 2 } s \right)  }  }{ \sqrt { 3 }  } \end{matrix}
\end{equation*}
where $ \alpha\left( s \right) =\left( {  \alpha }_{ 1 }\left( s \right) ,{  \alpha }_{ 2 }\left( s \right) ,{  \alpha }_{ 3 }\left( s \right)  \right) $ and $a=\left( 0,0,0\right) $
\end{example}

\begin{figure}[h!]
  \centering
    \includegraphics[width=0.3\textwidth]{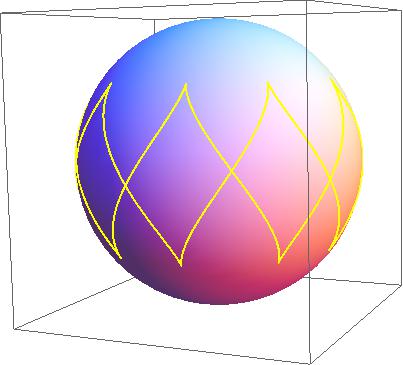}
    \includegraphics[width=0.3\textwidth]{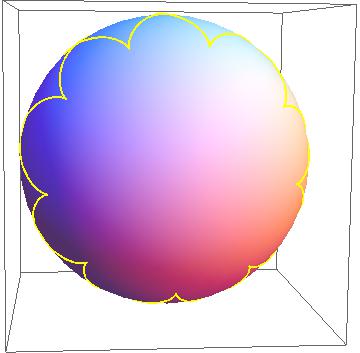}
    \caption{Spherical Helice}
\end{figure}

Now, we can write a new thorem about (7) in which we are looking for the $slant$ $helix$ condition of the curve $c$.

\begin{theorem}
Let $\gamma \left( s \right) $ be a unit speed spherical curve on $S^2$; $b, m, n $ be constant numbers; and a be a constant vector. The $geodesic$ $curvature$ of $\gamma \left( s \right)$ satisfies
\begin{equation*}
{ { { k }_{ g } }^{ 2 } }\left( s \right) =\frac { \left( ms+n \right) ^{ 2 } }{ 1-{ \left( ms+n \right)  }^{ 2 } } 
\end{equation*}
if and only if 
\begin{equation*}
c\left( s \right) =b\int { { e }^{ \int { k\left( s \right) ds }  } \gamma \left( s \right)  } ds+a
\end{equation*}
is a slant helix.
\end{theorem}

\begin{proof}
Let, for $\gamma$
\begin{equation}
{ { { k }_{ g } }^{ 2 } }\left( s \right) =\frac { \left( ms+n \right) ^{ 2 } }{ 1-{ \left( ms+n \right)  }^{ 2 } }.
\end{equation}
 For the curve $c\left( s \right)$ we have
\begin{equation*}
\begin{matrix} 
{ c }^{ ' }\left( s \right) =b{ e }^{ \int { k\left( s \right) ds }  }\gamma \left( s \right)  \\ 
{ c }^{ '' }\left( s \right) =b{ e }^{ \int { k\left( s \right) ds }  }\left\{ k\left( s \right) \gamma \left( s \right) +{ \gamma  }^{ ' }\left( s \right)  \right\}  \\
{ c }^{ ''' }\left( s \right) =b{ e }^{ \int { k\left( s \right) ds }  }\left\{ \left( { k }^{ 2 }\left( s \right) +{ k }^{ ' }\left( s \right)  \right) \gamma \left( s \right) +2k\left( s \right) { \gamma  }^{ ' }\left( s \right) +{ \gamma  }^{ '' }\left( s \right)  \right\} \\
\kappa\left( s \right) =\frac { 1 }{ b{ e }^{ \int { k\left( s \right)ds }  } }  \\
\tau\left( s \right) =\frac { { k }_{ g }\left( s \right) }{ b{ e }^{ \int { k\left( s \right)ds }  } }\\
\nu\left( s \right)= b{ e }^{ \int { k\left( s \right)ds }  }.
\end{matrix}
\end{equation*}
The geodesic curvature  of the spherical image of the principal normal indicatrix of $c$ is as follows
\begin{eqnarray*}
&&\sigma (s)=\left( \frac { { \kappa  }^{ 2 } }{ { { \nu \left( { \kappa  }^{ 2 }+{ \tau  }^{ 2 } \right)  } }^{ { 3 }/{ 2 } } } { \left( \frac { \tau  }{ \kappa  }  \right)  }^{ ' } \right) \left( s \right)\\
&&\quad\quad=\left( \frac { \frac { 1 }{ { \nu  }^{ 2 } }  }{ { { \nu \left( \frac { 1 }{ { \nu  }^{ 2 } } +\frac { { { k }_{ g } }^{ 2 } }{ { \nu  }^{ 2 } }  \right)  } }^{ { 3 }/{ 2 } } } { { k }_{ g } }^{ ' } \right) \left( s \right).
\end{eqnarray*}
So we have
\begin{equation}
\sigma (s)=\frac { { { k }_{ g } }^{ ' }\left( s \right)  }{ { \left( { { k }_{ g } }^{ 2 }\left( s \right) +1 \right)  }^{ { 3 }/{ 2 } } }
\end{equation}
Now, let's take $u\left( s \right) =ms+n$ then we have (11)
\begin{equation}
{ { { k }_{ g } }^{ 2 } }\left( s \right) =\frac { u^{ 2 }\left( s \right)  }{ 1-u^{ 2 }\left( s \right)  }. 
\end{equation}
If we take the derivates of the both sides of (13) for $s$ we have
\begin{eqnarray*}
&&2{ k }_{ g }{ { \left( s \right) k }_{ g } }^{ ' }\left( s \right) =\left( \frac { 2u{ u }^{ ' }\left( 1-{ u }^{ 2 } \right) -\left( -2u{ u }^{ ' } \right) { u }^{ 2 } }{ { \left( 1-{ u }^{ 2 } \right)  }^{ 2 } }  \right) \left( s \right) \\
&&{ k }_{ g }{ { \left( s \right) k }_{ g } }^{ ' }\left( s \right) =\left( \frac { u{ u }^{ ' } }{ { \left( 1-{ u }^{ 2 } \right)  }^{ 2 } }  \right) \left( s \right) \\
\end{eqnarray*}
\begin{equation}
{ { k }_{ g } }^{ ' }\left( s \right) =\left(\left( \frac { u{ u }^{ ' } }{ { \left( 1-{ u }^{ 2 } \right)  }^{ 2 } }  \right) \left( \varepsilon \sqrt { \frac { 1-{ u }^{ 2 } }{ { u }^{ 2 } }  }  \right)\right) \left( s \right)
\end{equation}
where $\varepsilon=\pm 1$. Putting (13) and  (14) in (12), we have
\begin{eqnarray*}
&&\sigma (s)=\frac { { { k }_{ g } }^{ ' }\left( s \right)  }{ { \left( { { k }_{ g } }^{ 2 }\left( s \right) +1 \right)  }^{ { 3 }/{ 2 } } }\\
&&\quad\quad= \left(\varepsilon \frac { \frac { \sqrt { 1-{ u }^{ 2 } } u{ u }^{ ' } }{ \left| u \right| { \left( 1-{ u }^{ 2 } \right)  }^{ 2 } }  }{ { \left( \frac { { u }^{ 2 } }{ 1-{ u }^{ 2 } } +1 \right)  }^{ { 3 }/{ 2 } } }  \right) \left( s \right)\\
&&\quad\quad=\left( \varepsilon \frac { \sqrt { 1-{ u }^{ 2 } } u{ u }^{ ' } }{ \left| u \right| { \left( 1-{ u }^{ 2 } \right)  }^{ 2 } } { \left( 1-{ u }^{ 2 } \right)  }^{ { 3 }/{ 2 } } \right) \left( s \right) \\
&&\quad\quad=\left( \varepsilon \frac { { \left( 1-{ u }^{ 2 } \right)  }^{ 2 } }{ { \left( 1-{ u }^{ 2 } \right)  }^{ 2 } } \frac { u }{ \left| u \right|  } { u }^{ ' } \right) \left( s \right) \\
&&\quad\quad=\varepsilon \frac { ms+n }{ \left| ms+n \right|  } m\\
&&\quad\quad=\varepsilon m
\end{eqnarray*}
which is constant.

Conversely, let $c\left( s \right)$ be a $slant$ $helix$, then the geodesic curvature  of the spherical image of the principal normal indicatrix of $c$ is a constant function. So we can take
\begin{equation*}
\sigma (s)=\left( \frac { { \kappa  }^{ 2 } }{ { { \nu \left( { \kappa  }^{ 2 }+{ \tau  }^{ 2 } \right)  } }^{ { 3 }/{ 2 } } } { \left( \frac { \tau  }{ \kappa  }  \right)  }^{ ' } \right) \left( s \right)=m
\end{equation*}
where $m \in R$. Therefore, from (12)
\begin{eqnarray*}
&&m=\left( \frac { { \kappa  }^{ 2 } }{ { { \nu \left( { \kappa  }^{ 2 }+{ \tau  }^{ 2 } \right)  } }^{ { 3 }/{ 2 } } } { \left( \frac { \tau  }{ \kappa  }  \right)  }^{ ' } \right) \left( s \right)\\
&&\quad=\frac { { { k }_{ g } }^{ ' }\left( s \right)  }{ { { \left( { { k }_{ g } }^{ 2 }\left( s \right) +1 \right)  } }^{ { 3 }/{ 2 } } }
\end{eqnarray*}
If we solve this differential equation, we have
\begin{equation*}
\frac { { k }_{ g }\left( s \right)  }{ \sqrt { { { k }_{ g } }^{ 2 }\left( s \right) +1 }  } =ms+n
\end{equation*}
where $n \in R$.
Then,
\begin{eqnarray*}
&&\frac { { { k }_{ g } }^{ 2 }\left( s \right)  }{ { { k }_{ g } }^{ 2 }\left( s \right) +1 } ={ \left( ms+n \right)  }^{ 2 }\\
&&\frac { { { k }_{ g } }^{ 2 }\left( s \right) +1-1 }{ { { k }_{ g } }^{ 2 }\left( s \right) +1 } ={ \left( ms+n \right)  }^{ 2 }\\
&&1-\frac { 1 }{ { { k }_{ g } }^{ 2 }\left( s \right) +1 } ={ \left( ms+n \right)  }^{ 2 }\\
&&{ { k }_{ g } }^{ 2 }\left( s \right) =\frac { 1 }{ { 1-\left( ms+n \right)  }^{ 2 } } -1\\
&&{ { { k }_{ g } }^{ 2 } }\left( s \right) =\frac { \left( ms+n \right) ^{ 2 } }{ 1-{ \left( ms+n \right)  }^{ 2 } }.
\end{eqnarray*}
\end{proof}
Furthermore, we can give a similar theorem for (8),

\begin{theorem}
Let $\gamma \left( s \right) $ be a unit speed spherical curve on $S^2$; $b, m, n, \theta $ be constant numbers; and a be a constant vector. The $geodesic$ $curvature$ of $\gamma \left( s \right)$ satisfies
\begin{equation*}
{ { { k }_{ g } }^{ 2 } }\left( s \right) =\frac { \left( ms+n \right) ^{ 2 } }{ 1-{ \left( ms+n \right)  }^{ 2 } } 
\end{equation*}
if and only if 
\begin{equation*}
c\left( s \right) =b\int _{ { s }_{ 0 } }^{ s }{ \gamma \left( \varphi  \right) d } \varphi +b\cot { \theta  } \int _{ { s }_{ 0 } }^{ s }{ p\left( \varphi  \right) d } \varphi +a
\end{equation*}
is a slant helix.
\end{theorem}

\begin{proof}
Let, for $\gamma$
\begin{equation}
{ { { k }_{ g } }^{ 2 } }\left( s \right) =\frac { \left( ms+n \right) ^{ 2 } }{ 1-{ \left( ms+n \right)  }^{ 2 } }.
\end{equation}
 For the curve $c\left( s \right)$ we have
\begin{equation*}
\begin{matrix} 
{ c }^{ ' }\left( s \right) =b\left( \gamma \left( s \right) +\cot { \theta  } p\left( s \right)  \right) \\
{ c }^{ '' }\left( s \right) =b\left( 1-\cot { \theta  } { k }_{ g }\left( s \right)  \right) p\left( s \right) \\ 
{ c }^{ ''' }\left( s \right) =-b\cot { \theta  } { { k }_{ g } }^{ ' }\left( s \right) p\left( s \right) +b\left( 1-\cot { \theta  } { k }_{ g }\left( s \right)  \right) \left( -\gamma \left( s \right) +{ k }_{ g }\left( s \right) p\left( s \right)  \right) \\
\kappa \left( s \right) =\varepsilon \frac { \sin ^{ 2 }{ \theta  } \left( 1-\cot { \theta  } { k }_{ g }\left( s \right)  \right)  }{ b } \\
\tau \left( s \right) =\frac { \sin ^{ 2 }{ \theta  } \left( { k }_{ g }\left( s \right) +\cot { \theta  }  \right)  }{ b } \\ 
\nu \left( s \right) =\varepsilon b\csc { \theta  }.
\end{matrix}
\end{equation*}
 where $\varepsilon=\pm 1$. 

The geodesic curvature  of the spherical image of the principal normal indicatrix of $c$ is as follows
\begin{eqnarray*}
&&\sigma (s)=\left( \frac { { \kappa  }^{ 2 } }{ { { \nu \left( { \kappa  }^{ 2 }+{ \tau  }^{ 2 } \right)  } }^{ { 3 }/{ 2 } } } { \left( \frac { \tau  }{ \kappa  }  \right)  }^{ ' } \right) \left( s \right) \\
&&\quad \quad =\left( \frac { \varepsilon \sin ^{ 3 }{ \theta  } { { k }_{ g } }^{ ' } }{ { { { b }^{ 3 }\left( \frac { \left( { \varepsilon  }^{ 2 }+\cot ^{ 2 }{ \theta  } -2\left( -1+{ \varepsilon  }^{ 2 } \right) \cot { \theta  } { k }_{ g }+\left( 1+{ \varepsilon  }^{ 2 }\cot ^{ 2 }{ \theta  }  \right) { { k }_{ g } }^{ 2 } \right) \sin ^{ 4 }{ \theta  }  }{ { a }^{ 2 } }  \right)  } }^{ { 3 }/{ 2 } } }  \right) \left( s \right) \\ 
&&\quad \quad=\left( \frac { \varepsilon { { k }_{ g } }^{ ' } }{ { { \sin ^{ 3 }{ \theta  } \left( \left( 1+\cot ^{ 2 }{ \theta  }  \right) \left( 1+{ { k }_{ g } }^{ 2 } \right)  \right)  } }^{ { 3 }/{ 2 } } }  \right) \left( s \right) \\ 
&&\quad \quad=\left( \frac { \varepsilon { { k }_{ g } }^{ ' } }{ { { \sin ^{ 3 }{ \theta  } \left( \frac { 1 }{ \sin ^{ 2 }{ \theta  }  } \left( 1+{ { k }_{ g } }^{ 2 } \right)  \right)  } }^{ { 3 }/{ 2 } } }  \right) \left( s \right) \\ 
&&\quad \quad=\left( \frac { \varepsilon { { k }_{ g } }^{ ' } }{ { { \left( 1+{ { k }_{ g } }^{ 2 } \right)  } }^{ { 3 }/{ 2 } } }  \right) \left( s \right) 
\end{eqnarray*}
So we have
\begin{equation}
\sigma (s)=\frac { \varepsilon { { k }_{ g } }^{ ' }\left( s \right)  }{ { \left( { { k }_{ g } }^{ 2 }\left( s \right) +1 \right)  }^{ { 3 }/{ 2 } } }
\end{equation}
Now, let's take $u\left( s \right) =ms+n$ then we have (15)
\begin{equation}
{ { { k }_{ g } }^{ 2 } }\left( s \right) =\frac { u^{ 2 }\left( s \right)  }{ 1-u^{ 2 }\left( s \right)  }. 
\end{equation}
If we take the derivates of the both sides of (17) for $s$ we have
\begin{eqnarray*}
&&2{ k }_{ g }{ { \left( s \right) k }_{ g } }^{ ' }\left( s \right) =\left( \frac { 2u{ u }^{ ' }\left( 1-{ u }^{ 2 } \right) -\left( -2u{ u }^{ ' } \right) { u }^{ 2 } }{ { \left( 1-{ u }^{ 2 } \right)  }^{ 2 } }  \right) \left( s \right) \\
&&{ k }_{ g }{ { \left( s \right) k }_{ g } }^{ ' }\left( s \right) =\left( \frac { u{ u }^{ ' } }{ { \left( 1-{ u }^{ 2 } \right)  }^{ 2 } }  \right) \left( s \right) \\
\end{eqnarray*}
\begin{equation}
{ { k }_{ g } }^{ ' }\left( s \right) =\left(\left( \frac { u{ u }^{ ' } }{ { \left( 1-{ u }^{ 2 } \right)  }^{ 2 } }  \right) \left( \varepsilon \sqrt { \frac { 1-{ u }^{ 2 } }{ { u }^{ 2 } }  }  \right)\right) \left( s \right).
\end{equation}
Putting (17) and  (18) in (16), we have
\begin{eqnarray*}
&&\sigma (s)=\frac {  \varepsilon{ { k }_{ g } }^{ ' }\left( s \right)  }{ { \left( { { k }_{ g } }^{ 2 }\left( s \right) +1 \right)  }^{ { 3 }/{ 2 } } }\\
&&\quad\quad= \left({\varepsilon}^2 \frac { \frac { \sqrt { 1-{ u }^{ 2 } } u{ u }^{ ' } }{ \left| u \right| { \left( 1-{ u }^{ 2 } \right)  }^{ 2 } }  }{ { \left( \frac { { u }^{ 2 } }{ 1-{ u }^{ 2 } } +1 \right)  }^{ { 3 }/{ 2 } } }  \right) \left( s \right)\\
&&\quad\quad=\left( \frac { \sqrt { 1-{ u }^{ 2 } } u{ u }^{ ' } }{ \left| u \right| { \left( 1-{ u }^{ 2 } \right)  }^{ 2 } } { \left( 1-{ u }^{ 2 } \right)  }^{ { 3 }/{ 2 } } \right) \left( s \right) \\
&&\quad\quad=\left(\frac { { \left( 1-{ u }^{ 2 } \right)  }^{ 2 } }{ { \left( 1-{ u }^{ 2 } \right)  }^{ 2 } } \frac { u }{ \left| u \right|  } { u }^{ ' } \right) \left( s \right) \\
&&\quad\quad= \frac { ms+n }{ \left| ms+n \right|  } m\\
&&\quad\quad= \varepsilon m
\end{eqnarray*}
which is constant.

Conversely, let $c\left( s \right)$ be a $slant$ $helix$, then the geodesic curvature  of the spherical image of the principal normal indicatrix of $c$ is a constant function. So we can take
\begin{equation*}
\sigma (s)=\left( \frac { { \kappa  }^{ 2 } }{ { { \nu \left( { \kappa  }^{ 2 }+{ \tau  }^{ 2 } \right)  } }^{ { 3 }/{ 2 } } } { \left( \frac { \tau  }{ \kappa  }  \right)  }^{ ' } \right) \left( s \right)=m
\end{equation*}
where $m \in R$. Therefore, from (16)
\begin{eqnarray*}
&&m=\left( \frac { { \kappa  }^{ 2 } }{ { { \nu \left( { \kappa  }^{ 2 }+{ \tau  }^{ 2 } \right)  } }^{ { 3 }/{ 2 } } } { \left( \frac { \tau  }{ \kappa  }  \right)  }^{ ' } \right) \left( s \right)\\
&&\quad=\frac { \varepsilon{ { k }_{ g } }^{ ' }\left( s \right)  }{ { { \left( { { k }_{ g } }^{ 2 }\left( s \right) +1 \right)  } }^{ { 3 }/{ 2 } } }
\end{eqnarray*}
If we solve this differential equation, we have
\begin{equation*}
\frac { \varepsilon { k }_{ g }\left( s \right)  }{ \sqrt { { { k }_{ g } }^{ 2 }\left( s \right) +1 }  } =ms+n
\end{equation*}
where $n \in R$.
Then,
\begin{eqnarray*}
&&\frac { { { k }_{ g } }^{ 2 }\left( s \right)  }{ { { k }_{ g } }^{ 2 }\left( s \right) +1 } ={ \left( ms+n \right)  }^{ 2 }\\
&&\frac { { { k }_{ g } }^{ 2 }\left( s \right) +1-1 }{ { { k }_{ g } }^{ 2 }\left( s \right) +1 } ={ \left( ms+n \right)  }^{ 2 }\\
&&1-\frac { 1 }{ { { k }_{ g } }^{ 2 }\left( s \right) +1 } ={ \left( ms+n \right)  }^{ 2 }\\
&&{ { k }_{ g } }^{ 2 }\left( s \right) =\frac { 1 }{ { 1-\left( ms+n \right)  }^{ 2 } } -1\\
&&{ { { k }_{ g } }^{ 2 } }\left( s \right) =\frac { \left( ms+n \right) ^{ 2 } }{ 1-{ \left( ms+n \right)  }^{ 2 } } .
\end{eqnarray*}
\end{proof}

\label{sec:resu}

\end{document}